\documentclass[12pt]{amsart}
\usepackage[cp1250]{inputenc}
\usepackage[T1]{fontenc}
\usepackage{amscd}

\usepackage{amsmath}
\usepackage{amsthm}
\usepackage{amssymb}
\usepackage{amscd}
\usepackage{graphicx}
\usepackage{epsfig}
\usepackage{bbm}

\theoremstyle{plain}
\newtheorem{thm}{Theorem}[section]
\newtheorem{lem}[thm]{Lemma}
\newtheorem{prop}[thm]{Proposition}
\newtheorem{cor}[thm]{Corollary}
\newtheorem{con}[thm]{Conjecture}

\theoremstyle{definition}
\newtheorem{dff}[thm]{Definition}
\newtheorem{rem}[thm]{Remark}
\newtheorem{exa}[thm]{Example}

\numberwithin{equation}{section}
\def\B{\mathcal{B}}

\def\D{\mathcal{D}}

\def\H{\mathcal{H}}

\def\P{\mathcal{P}}

\def\U{\mathcal{U}}
\def\V{\mathcal{V}}

\def\Q{\mathcal{Q}}

\def\phi{\varphi}
\def\p{\partial}

\def\hr{\H_c(\p\t\R_+)}
\def\hrn{\H_c(\R^n)}

\def\rz{\mathbb{R}}
\def\R{\rz}

\def\N{\mathbb{N}}

\def\wyr#1{\textit{#1}}
\def\Z{\mathbb{Z}}
\def\s{\subset}
\def\hm{\H_c(M)}
\def\hdm{\H_c^{\p}(M)}

\def\t{\times}
\def\r{\rightarrow}

\def\ld{,\ldots,}
\def\pm{\p_M}

 \DeclareMathOperator{\diff}{Diff}

 \DeclareMathOperator{\cl}{cl}
  
  \DeclareMathOperator{\conj}{conj}
 
 \DeclareMathOperator{\frag}{frag}
 
\DeclareMathOperator{\id}{id} \DeclareMathOperator{\jac}{Jac}
 \DeclareMathOperator{\supp}{supp}

\thanks{Partially supported by the Polish Ministry of Science and Higher Education and the
AGH grant n. 11.420.04}

\address{Faculty of Applied Mathematics, AGH University of Science and
\linebreak Technology, al. Mickiewicza 30, 30-059 Krak\'ow,
Poland} \email{kowalik@wms.mat.agh.edu.pl,tomasz@uci.agh.edu.pl}
\date{March 18, 2011}

\title{On the  homeomorphism  groups of manifolds and their universal coverings }
\author{Agnieszka Kowalik,  Tomasz Rybicki}

\begin{document}

\maketitle

\begin{abstract}
Let $\H_c(M)$  stand for the path connected identity component of
the group of all compactly supported homeomorphisms of a manifold
$M$. It is shown that $\H_c(M)$ is perfect and simple under mild
assumptions on $M$.   Next, conjugation-invariant norms on
$\H_c(M)$ are considered and the boundedness of $\H_c(M)$ and its
subgroups is investigated.
 Finally, the structure of the universal covering group of
$\H_c(M)$ is studied.
\end{abstract}

\section{Introduction}

Let $M$ be a topological metrizable manifold of dimension $n\geq
1$, possibly with boundary, and let $\H(M)$ (resp. $\H_c(M)$) be
the path connected identity component of the group of all (resp.
compactly supported) homeomorphisms of a manifold $M$ endowed with
the compact-open topology. In this paper we will deal with
algebraic properties of the group $\H_c(M)$ and  of its universal
covering.

 Recall that a group $G$ is
called \emph{perfect} if it is equal to its own commutator
subgroup $[G,G]$. That is, $H_1(G)=0$. The following basic fact is
probably well-known but we have not found it explicitly proven in
the literature.

\begin{thm} Assume that either $M$ is compact (possibly with
boundary), or $M$  admits a compact exhaustion, i.e. there is a
sequence of compact submanifolds with boundary
$(M_i)_{i=1}^{\infty}$ with $\dim M_i=\dim M=n$ such that $M_1\s
M^o_2\s M_2\s\ldots$ and $M=\bigcup_{i=1}^{\infty}M_i$. Then the
group $\H_c(M)$ is perfect.
\end{thm}

The proof of the perfectness is a consequence of Mather's paper
\cite{Mat71} combined with Edwards and Kirby \cite{ed-ki},
Corollary 1.3. In the case $n=1$ and $M$ with boundary the proof
requires an additional argument. See section 3. A special case of
Theorem 1.1 was already proved by Fisher \cite{fis} (see also
Anderson \cite{An}). Observe that McDuff in \cite{MD} proved that
$\H(M)$ is perfect provided $M$ is the interior of a compact
manifold with boundary. There exist some generalizations of
Theorem 1.1 (see, e.g., Fukui and Imanishi \cite{fu-im}, and
Rybicki \cite{ry3}).

If $M$ is a smooth manifold then Theorem 1.1 has its smooth
analogue. Let $\D(M)$ be the identity component of the group of
all compactly supported $C^{\infty}$-diffeomorphisms of $M$.
Thurston proved that  $\D(M)$ is perfect and simple (see
\cite{Thu74}, \cite{ban2}). Also Mather in \cite{Mat74} proved the
same in the class of $C^r$-diffeomorphisms unless $r=\dim M+1$.
Analogous results for classical groups of diffeomorphisms are also
known (\cite{ban1}, \cite{ban2}, \cite{HR}, \cite{ry6}).

In the case of a manifold with boundary $M$ we denote by $M^o$ the
interior of $M$, and by $\pm$ the boundary of $M$. We will
consider the following groups:
$$ \H_c(M^o)\leq\H_c^{\p}(M)\leq\H_c(M)\leq\H(M^o).$$
Here $h\in\H_c^{\p}(M)$ if there is a compactly supported isotopy
$h_t$ connecting $h_0=\id$ with $h_1=h$ such that $h_t=\id$ on
$\pm$ for all $t$. Moreover, $\hm$ identifies with a subgroup of
$\H(M^o)$ by restricting elements of $\hm$ to $M^o$.

\begin{thm}
If the boundary $\pm$ is compact then $\hdm$ is a perfect group.
\end{thm}

 Concerning the simplicity of $\H_c(M)$ we have the following

\begin{cor}
Let $M$ be connected and satisfy the hypothesis of Theorem 1.1.
Then $M$ is boundaryless (i.e. $\pm=\emptyset$) if and only if
$\H_c(M)$ is simple.
\end{cor}

The proof will be given in section 4 together with  further
comments on the simplicity by using some ideas of Ling \cite{li}.

 Conjugation-invariant norms related to homeomorphism groups on $M$ are considered in section 5.
Recall that a group is \emph{bounded} if every
conjugation-invariant norm  is bounded on it.
 Following an argument from \cite{BIP}  we will prove in section 6
 the following
 \begin{thm}
 Under the assumption of Theorem 1.1 on $M$,
 $\H_c(M)$ is bounded if and only if $\frag_{M}$ is bounded,
 where $\frag_M$ is the fragmentation norm on $M$
 with respect to homeomorphisms.
 In particular, $\H_c(\R^n)$ is bounded.
 \end{thm}

 Also we have the following boundedness theorem.

\begin{thm}
Let $\pm$ be compact. If the group $\H_c(M)$ is bounded then
$\H(\pm)$ is bounded also. Moreover, if the group $\H_c(M^o)$ is
bounded then so is the group $\hdm$.
\end{thm}

Observe that $\H_c(M^o)$ is bounded if $M^o$ is portable (sect.
6). In \cite{ry6} the second-named author proved that $\H(M^o)$ is
bounded provided  so is $\H_c(M^o)$.

The last part of the paper is devoted to the structure of the
universal coverings of some homeomorphism groups. Let
$\H_c(M)^{\sim}$  denote the universal covering group of
$\H_c(M)$.

 \begin{thm} Let $n=\dim M\geq 2$ or $\pm=\emptyset$. The  group
 $\H_c(M)^{\sim}$
 is perfect. Moreover, the groups $\H_c(\R^n)^{\sim}$ and $\H_c(\R^n_+)^{\sim}$
   are acyclic, where $\R^n_+=\{x\in\R^n: x_n\geq 0\}$ is the half-space.
\end{thm}

The proof will be given in section 7.  We will also study the
problem of boundedness.
\begin{thm}
 Let $\frag^{iso}_M$ be the isotopy fragmentation norm on the universal covering group $\H_c(M)^{\sim}$
 (c.f. sect.7). Suppose that $\dim M\geq 2$ or $\pm=\emptyset$.
 Then $\H_c(M)^{\sim}$ is bounded if and only if $\frag^{iso}_{M}$ is bounded.
 \end{thm}

We emphasize that many facts presented in this paper are specific
for the topological category, that is there are no longer true in
the smooth (or even Lipschitz) category. See, e.g., Remark 3.4 and
Prop. 6.3.

\section{Fragmentation property and isotopy extension theorem}

  The results of this paper depend essentially on the deformation properties
for the spaces of imbeddings obtained by Edwards and Kirby in
\cite{ed-ki}. See also Siebenmann \cite{sie}. Let us recall basic
notions and facts from \cite{ed-ki}.

 Given a subset $S\s M$, by
$\H_S(M)$  we denote the path connected identity component of the
subgroup of all elements of $\H(M)$ with compact support contained
in $S$. By a \emph{ball} (resp. \emph{half-ball}) $B$ we mean rel.
compact open ball (resp. half-ball with $\p_B=B\cap\p_M$) embedded
in $M$ with its closure. By $\B$ we denote the family of all balls
and half-balls in $M$.

Using the Alexander trick, we have that $\H(\R^n)$ coincides with
the group of all compactly supported homeomorphisms of $\R^n$. In
fact, if $\supp(g)$ is compact, we define an isotopy $g_{t} :
\mathbb{R}^{n} \rightarrow \mathbb{R}^{n}$, $t \in I$, from the
identity to $g$, by

\begin{equation*}
g_{t}(x)= \left\{
\begin{array}{lcl}
tg\left( \frac{1}{t}x \right)& for & t>0\\
x&for&t=0.
\end{array} \right.
\end{equation*}

In particular, for every ball $B$ in $M$ the group $\H_B(M)$
consists of all homeomorphisms compactly supported in $B$. Observe
that the Alexander trick is no longer true in the smooth category.

Let us formulate the fragmentation property in the following
stronger way.
\begin{dff}  Let  $\U$ be an open  covering of $M$.
 A subgroup $G\leq \H(M)$ is  \emph{locally continuously factorizable}
  if  for any finite subcovering $(U_i)_{i=1}^d$ of $\B$, there
 exist a neighborhood $\P$ of $\id\in G$ and continuous mappings $\sigma_i:\P\r G$, $i=1\ld d$,
  such that   for all $f\in
 \P$ one has \begin{equation*}f=\sigma_1(f)\ldots \sigma_d(f),\quad \supp(\sigma_i(f))\s U_i, \forall i.
 \end{equation*}
 \end{dff}

 Throughout for a topological group $G$ by $\mathcal{P}G$ we will denote
the totality of paths  $\gamma:I\r G$ with $\gamma(0)=e$ (where
$I=[0,1]$). Observe that Def. 2.1 can also be formulated for $\P
G$ rather than $G$, where $G\leq\H(M)$.

 From now on $M$
is a metrizable topological manifold. If $U$ is a subset of $M$, a
\emph{proper imbedding} of $U$ into $M$ is an imbedding $h: U
\rightarrow M$ such that $h^{-1}(\partial M)=U \cap \partial M$.
An \emph{isotopy} of $U$ into $M$ is a family of imbeddings
$h_{t}: U \rightarrow M$, $t \in I$, such that the map $h: U
\times I \rightarrow M$ defined by $h(x,t)=h_{t}(x)$ is
continuous. An isotopy is \emph{proper} if each imbedding in it is
proper. Now let $C$ and $U$ be  subsets of $M$ with $C\subseteq
U$. By $I(U,C;M)$ we denote the space of proper imbeddings of $U$
into $M$
 which equal the identity on $C$, endowed with the compact-open topology.

Suppose $X$ is a space with subsets $A$ and $B$. A
\emph{deformation of A into B}
 is a continuous mapping $\varphi : A \times I\rightarrow X$ such that $\varphi|_{A\times 0}=\id_{A}$
  and $\varphi(A\times 1) \subseteq B$. If $\P$ is a subset of
  $I(U;M)$ and $\varphi:\P \times I\rightarrow I(U;M)$ is a deformation of $\P$,
  we may equivalently view $\varphi$ as a map $\varphi:\P\times I\times U\rightarrow M$
  such that for each $h\in \P$ and $t\in I$, the map $\varphi(h,t):U\rightarrow M$ is
  a proper imbedding.

If $W\subseteq U$, a deformation $\varphi: \P\times I\rightarrow
I(U;M)$ is \emph{modulo W} if $\varphi(h,t)|_{W}=h|_{W}$ for all
$h\in \P$ and $t\in I$.

 Suppose $\varphi: \P\times I \rightarrow
I(U;M)$ and $\psi: \Q\times I \rightarrow I(U;M)$ are deformations
of subsets of $I(U;M)$ and suppose that $\varphi(\P\times
1)\subseteq \Q$. Then the \emph{composition} of $\psi$ with
$\varphi$,  denoted by $\psi \star \varphi$, is the deformation
$\psi \star \varphi :\P\times I\rightarrow I(U;M)$ defined by
\begin{equation*}
\psi \star\varphi(h,t)= \left\{
\begin{array}{lcl}
\varphi(h,2t)& for & t\in [0,1/2]\\
\psi(\varphi(h,1),2t-1)& for &t\in [1/2,1].
\end{array}
\right.
\end{equation*}

The main result of \cite{ed-ki} is the following

\begin{thm}
Let $M$ be a topological manifold and let $U$ be a neighborhood in
$M$ of a compact subset $C$. For any neighborhood $\Q$ of the
inclusion $i:U\subset M$ in $I(U;M)$ there are a neighborhood $\P$
of $i\in I(U;M)$ and a deformation $\varphi:\P\t I\r \Q$ into
$I(U,C;M)$ which is modulo  the complement of a compact
neighborhood of $C$ in $U$ and such that $\varphi(i,t)=i$ for all
$t$. We have also that if $D_i\subset V_i$, $i=1\ld q$, is a
finite family of closed subsets $D_i$ with their neighborhoods
$V_i$, then $\varphi$ can be chosen so that the restriction of
$\varphi$ to $(\P\cap I(U,U\cap V_i;M))\t I$ assumes its values in
$I(U,U\cap D_i;M)$ for each $i$.

Moreover, if $M$ has compact boundary $\pm$ then $\varphi$
restricted to $(\P\cap I(U,\pm\cap U; M))\t I$ takes its values
into $I(U,\pm\cap U; M)$.
\end{thm}

The first part coincides with Theorem 5.1\cite{ed-ki}. The second
part is specified in Remark 7.2 in \cite{ed-ki}.

We can derive from Theorem 2.2 the following fragmentation
theorem.
\begin{thm} Let $M$ be a compact manifold, possibly with boundary.
Then the groups $\hm$, $\hdm$, $\P\hm$ and $\P\hdm$ are locally
continuously factorizable, i.e. they satisfy Def. 2.1.
\end{thm}

\begin{proof} (See also \cite{ed-ki}.) We will consider only the case
of $\hm$, the remaining ones being analogous. First we have to
shrink the cover $(U_i)_{i=1}^d$ $d$ times, that is we choose an
open $U_{i,j}$ for every $i=1\ld d$ and $j=0\ld d$ with
$U_{i,0}=U_i$ such that $\bigcup_{i=1}^dU_{i,j}=M$ for all $j$ and
such that $\cl(U_{i,j+1})\s U_{i,j}$ for all $i,j$.  We make use
of Theorem 2.2 $d$ times with $q=1$. Namely, for $i=1\ld d$ we
have a neighborhood $\P_i$ of the identity in
$I(M,\bigcup_{\alpha=1}^{i-1}U_{\alpha,i-1}; M)$ and a deformation
$\phi_i:\P_i\t I\r \H_c(M)$ which is modulo $M\setminus U_{i,0}$
and which takes its values in
$I(M,\bigcup_{\alpha=1}^{i}\cl(U_{\alpha,i}); M)$ and such that
$\phi_i(\id, t)=\id$ for all $t$. Here we apply Theorem 2.2 with
$C=\cl(U_{i,i})$, $U=U_{i,0}$,
$D_1=\bigcup_{\alpha=1}^{i-1}\cl(U_{\alpha,i})$ and
$V_1=\bigcup_{\alpha=1}^{i-1}U_{\alpha,i-1}$. Taking a
neighborhood $\P$ of id small enough, we have that
$\phi_d\star\cdots\star\phi_1$ restricted to $\P\t I$ is well
defined. For every $h\in\P$ we set $h_0=h$ and
$h_i=\phi_i\star\cdots\star\phi_1(h,1)$, $i=1\ld d$. It follows
that $h_d=\id$ and $h=\prod_{i=1}^dh_ih_{i-1}^{-1}$. It suffices
to define $\sigma_i:\P\r\H_c(M)$ by $\sigma_i(h)=h_ih_{i-1}^{-1}$
for all $i$.
\end{proof}

\begin{cor}\label{isot}
Let $h_{t}:M \rightarrow M$, $t \in I$, be an isotopy of a compact
manifold $M$ with $h_0=\id$, and let $(U_{i})_{i=1}^d$ be an open
cover of $M$. Then $h_{t} $ can be written as a composition of
isotopies $h_{t}=h_{k,t}h_{k-1,t}\ldots h_{1,t}$, where each
isotopy $h_{j,t}: M \rightarrow M$  is supported by some  $U_{i}$.
Moreover, if $h_t|_{\pm}=\id$ for all $t$, then
$h_{j,t}|_{\pm}=\id$ for all $j$ and $t$. The same is true for
homeomorphisms instead of isotopies.
\end{cor}

Another important consequence of Theorem 2.2 is the following
Isotopy Extension Theorem.
\begin{thm} \cite{ed-ki}
Let $f_t$ be an isotopy in $\H(M)$ and let $C\s M$ be a compact
set. Then for any open neighborhood $U$ of the track of $C$ by
$f_t$ given by $\bigcup_{t\in [0,1]}f_t(C)$ there is an isotopy
$g_t$ in $\H_c(M)$ such that $g_t=f_t$ on $C$ and $\supp(g_t)\s
U$.
\end{thm}

\section{Perfectness of $\H_c(M)$ and $\H^{\p}_c(M)$}

The goal of this section is to give the proof of Theorem 1.1. We
begin with the following fact, with a straightforward proof, which
plays a basic role in studies on homeomorphism groups.
\begin{lem}\cite{Mat71}(Basic lemma)
Let $B\s M$ be a ball and $U\s M$ be an open subset such that
$\overline{B}\s U$. Then there are $\phi\in\H_U(M)$ and a
homomorphism $S:\H_B(M)\r \H_U(M)$ such that $h=[S(h),\phi]$ for
all $h\in\H_B(M)$.

\end{lem}
\begin{proof} First choose a larger ball $B'$ such that $\overline{B}\s
B'\s\overline{B'}\s U$. Next, fix $p\in\p_{B'}$ and set $B_0=B$.
There exists a sequence of balls $(B_k)_{k=1}^{\infty}$ such that
$\cl(B_k)\s B'$ for all $k$, where the family
$(B_k)_{k=0}^{\infty}$ is pairwise disjoint, locally finite in
$B'$, and $B_k\r p$ as $k\r\infty$. Choose a homeomorphism
$\phi\in\H_U(M)$ such that $\varphi(B_{k-1})=B_k$ for
$k=1,2,\ldots$. Here we use the fact that $\H_U(M)$ acts
transitively on the family of balls in $B'$, c.f. \cite{hir}.

Next we define a homomorphism $S:\H_B(M)\r\H_U(M)$ by the formula
$$ S(h)=\phi^kh\phi^{-k}\quad\hbox{on}\,B_k,\ k=0,1,\ldots$$
and $S(h)=\id$ outside $\bigcup_{k=0}^{\infty}B_k$. It is clear
that $h=[S(h),\phi]$, as required.
\end{proof}

The above reasoning appeared in Mather's paper \cite{Mat71}.
Actually Mather proved also the acyclicity of $\H(\R^n)$. It is
easily seen that \cite{Mat71} and Lemma 3.1 are no longer true for
$C^1$ homeomorphisms. However, Tsuboi gave an excellent
improvement of this reasoning and adapted it for
$C^r$-diffeomorphisms with small $r$, see \cite{tsu1}.

\begin{cor} Assume that either
\begin{enumerate}\item $\pm\neq\emptyset$ with $\dim M\geq 2$, and $B, U\s
M$ are  such that $B$ is a half-ball, and $U$ is open with
$\overline{B}\s U$; or \item $M=N\t\R$, where $N$ is a manifold,
and $B=N\t I$, $U=N\t J$ where $I, J\s\R$ are open intervals with
the closure of $I$ contained in $J$.\end{enumerate} Then there are
$\phi\in\H_U(M)$ and a homomorphism $S:\H_B(M)\r \H_U(M)$ such
that $h=[S(h),\phi]$ for all $h\in\H_B(M)$. Moreover, in the case
(1), if $h\in\H_B(M)$ satisfies $h=\id$ on $\pm$ then $S(h)=\id$
on $\pm$.
\end{cor}

The proof is analogous to that of Lemma 3.1.

Suppose that  $\{U_i\}_{i\in \N}$ is a pairwise disjoint, locally
finite family of open sets of $M^o$. Put $U=\bigcup_iU_i$. Let
$\H_{[U]}(M)$ (resp. $\H_{[U]}^{\p}(M)$) denote the group of all
homeomorphisms from $\hm$ (resp. $\hdm$) supported in $U$ such
that for the decomposition $h=h_1h_2\ldots$ resulting from the
partition $U=\bigcup_iU_i$ one has $h_i\in\H_{U_i}(M)$ for all
$i$.
\begin{cor}
Let $\overline{B_i}\s U_i$, $i\in \N$, and let the pair
$(B_i,U_i)$ be such as in Lemma 3.1 or Corol. 3.2.  Then any
element
 $h\in\H_{[B]}(M)$ (where $B=\bigcup_iB_i$) is expressed as $h=\tilde h\bar h$, where
  $\tilde h,\bar h\in\H_{[U]}(M)$. Moreover, we can arrange so that if
  $h\in\H_{[B]}^{\p}(M)$ then $\tilde h,\bar h\in\H_{[U]}^{\p}(M)$.
\end{cor}

In fact, we can glue together $S(h_i)$ and $\phi_i$ obtained for
particular $U_i$.

\medskip
\noindent\emph{Proof of Theorem 1.1 for $n>1$ or $\pm=\emptyset$}.
For $M$ compact it follows from Corol. 2.4, Lemma 3.1,  Corol.
3.2(1) and, for $\pm \neq\emptyset$ and $\dim M=1$, from the fact
that $\H([0,1])$ is perfect. The proof of the latter fact will
follow from the proof of Theorem 1.2 below.  Suppose now that $M$
admits a compact exhaustion. If $h\in\H_c(M)$ then there are
$j\in\N$ and an isotopy $h_t$ such that $h_0=\id$, $h_1=h$, and
$\supp(h_t)\s M_j$ for all $t$. In view of Corol. 2.2 it follows
that $h|_{M_j}$ can be written  as $h|_{M_j}=h_d\ldots h_1$  such
that $h_i\in\H_{B_i}(M_j)$, where $B_i$ is a ball or half-ball of
$M_i$ for $i=1\ld d$. Moreover, we have $h_i=\id$ on $\p_{M_j}$
for all $i$. Then due to Corol. 3.2(1) we have
$h_i=[S_i(h_i),\phi_i]$ and $S_i(h_i)=\id$ on $\p_{M_j}$ for each
$i$. Is is easily seen that $\phi_i$ may be defined as an element
of $\H_c(M)$ supported in the interior of $M_{j+1}$. Thus
extending each $h_i$ to $M$ by putting $h_i=\id$ off $M_j$, the
perfectness of $\H_c(M)$ follows. \quad$\square$
\medskip

\noindent\emph{Proof of Theorem 1.1 for $n=1$ and
$\pm\neq\emptyset$, and of Theorem 1.2}. Let $M$ be a manifold
with boundary $\p$.  By a  \emph{collar} neighborhood of $\p$ we
mean a set $P=\p\t[0,1]$ embedded in $M$, where
 $\p\t\{0\}$ identifies with $\p$. It is well-known that such a
 neighborhood exists.

 In the case of 1.1 we have $\p=\{0\}$. In view of
Theorem 2.3 it suffices to consider $\hr$, where
$\R_+=[0,\infty)$. For any $f\in \hr$ there is a sequence of reals
from (0,1)
\begin{equation}1>b_1>\bar b_1>\bar a_1>a_1>b_2>\ldots>b_k>\bar
b_k>\bar a_k>a_k>\ldots>0,\end{equation} tending to 0, and $h\in
\hr$ such that
 \begin{equation}h=f\quad\hbox{ on}\quad\p\t\quad\bigcup_{k=1}^{\infty} [\bar a_k,\bar b_k].
\end{equation}
Moreover, setting $A_k:=\p\t(a_k,b_k)$ and
$A:=\bigcup_{k=1}^{\infty}A_k$, we may also have that
\begin{equation}\supp(h)\s A,\end{equation} and that
 for the  decomposition $h=h_1h_2\ldots$ resulting from the
partition  $A=\bigcup_{k=1}^{\infty}A_k$ and from (3.3) we have
\begin{equation}h_k\in \H_{A_k}(\p\t\R_+)\quad\hbox{for all}\;k.
\end{equation}

The condition (3.4) means that we exclude any twisting of $h_k$.

 In order to show the above statements we apply Theorem 2.5 for
 $M^o$.
This enables us to define recurrently $b_k>\bar b_k>\bar a_k>a_k$
and $h|_{\p\t[a_k,b_k]}$ for $k=1,2,\ldots$. In fact, let $f_t$ be
an isotopy in $\hr$ connecting $f$ with the identity. Suppose we
have defined $1>b_1>\ldots> a_{k-1}$ and $g\in\hr$ such that
$g=f\quad\hbox{on}\quad\p\t\bigcup_{i=1}^{k-1} [\bar a_i,\bar
b_i]$, $\supp(g)\s\bigcup_{i=1}^{k-1}A_i$, and
$g_i\in\H_{A_i}(\p\t\R_+)$ for all $i\leq k-1$. Now it suffices to
take $a_{k-1}>b_k>\bar b_k>\bar a_k>a_k$ in such a way that
$\p\t(0,b_k]$ is disjoint with
$\bigcup_{t\in[0,1]}f_t^{-1}(\p\t[a_{k-1},1])$. In view of Theorem
2.5 we get an isotopy $\bar h_t$ such that $\bar h_t=f_t$ on
$\p\t[\bar a_k,\bar b_k]$ and $\supp(\bar h_t)\s\p\t( a_k, b_k)$.
Next we define $h$ on $\p\t[a_k,1]$ by gluing together $g$ and
$\bar h_1$. Continuing this procedure we define
$h\in\H^{\p}(\p\t\R_+)$ fulfilling  (3.2), (3.3) and (3.4). Here
we put $h(x,0)=(x,0)$ for all $x\in\p$.

Next we set $h':=h^{-1}f$, that is $f=hh'$. It follows that $h'$
also enjoys the properties (3.2), (3.3) and (3.4) with a suitably
chosen sequence similar to (3.1).

Let $U_k= (\tilde a_k,\tilde b_k)$, where $\tilde a_k,\tilde
b_k\in(0,1)$, $k=1,2,\ldots$, are such that $\tilde a_{k-1}>\tilde
b_k>b_k>a_k>\tilde a_k$ for all $k$ ($\tilde b_0=1$). Now in view
of Corollaries 3.2 and 3.3 with $M=\p\t\R_+$ $h$ belongs to the
commutator subgroup of the group $\H_{[U]}^{\p}(\p\t\R_+)$, where
$U=\bigcup_k U_k$.  More precisely $h=[\tilde h,\bar h]$ for
$\tilde h,\bar h\in \H^{\p}_{[U]}(\p\t\R_+)$. It is easily seen
that $\tilde h,\bar h\in \H^{\p}_c(\p\t\R_+)$. The same is true
for $h'$. Thus $\H^{\p}_c(\p\t\R_+)$ is a perfect group. \quad
$\square$
\medskip

\begin{rem}
(1) Tsuboi gave another proof of the perfectness of $\H(\R_+)$ in
\cite{tsu2}. He did not  use \cite{ed-ki} in it.

(2) Given a  smooth manifold with boundary $M$ of dimension $\geq
2$, it is  known that the group $\D(M)$ is perfect (see Rybicki
\cite{Ryb98}; also Abe and Fukui \cite{AF01} by using a different
method).
 For $n=1$ $\D(M)$ is not perfect. In particular, Fukui in
 \cite{Fuk80} calculated that $H_1(\D(\R_+))=\R$.

(3) Let $\R^n_+=[0,\infty)\t\R^{n-1}$ be the half-space and $0\leq
s\leq\infty$. Let $\D_s(\R^n_+)$ be the compactly supported
identity component of the subgroup of all elements of $\D(\R^n_+)$
which are $s$-tangent to the identity on $\p_{\R^n_+}$. Here $f$
is 0-tangent to id means that $f=\id$ on $\p_{\R^n_+}$. If $0\leq
s<\infty$ then $\D_s(\R^n_+)$ is not perfect. In fact, for any
diffeomorphisms $f,g\in \D_s(\R^n_+)$ we have
$$     D^{s+1}(fg)(0)= D^{s+1}f(0)+D^{s+1}g(0),\quad    D^{s+1}f^{-1}(0) = -D^{s+1}f(0).  $$
Therefore if we choose $h\in \D_s(\R^n_+)$ such that
$D^{s+1}h(0)\not =0$, the above equalities yield that $h$ cannot
be in the commutator subgroup.
\end{rem}

Finally, let us indicate further perfectness result concerning
homeomorphism groups.  Let $M$ be a compact manifold with boundary
$\p$. Let $\p=\p_i$, $i=1\ld k$, be the family of all connected
components of the boundary $\p$ of $M$, that is
$\p=\p_1\cup\ldots\cup\p_k$. Let $K=\{1\ld k\}$. For any $J\s K$
let $\H(M^o,J)$ denote all the elements of $\H(M^o)$ that can be
joined with the identity by an isotopy which stabilizes near
$\p_J$, where $\p_J:=\bigcup_{i\in J}\p_i$. In particular,
$\H(M^o)=\H(M^o,\emptyset)$ and $\H_c(M^o)=\H(M^o,K)$. Then we
have \begin{thm}  \cite{MD} The groups $\H(M^o,J)$, where $J\s K$,
are perfect.
\end{thm}

For the proof, see also \cite{ry6}. The proof is  no longer valid
if we drop the assumption that $M^o$ is the interior of a manifold
with boundary, e.g. if $M$ is the cylinder $\mathbb S^1\t\R$ with
attached infinitely many handles.

\section{On the simplicity of $\H_c(M)$}
The following result is related to Ling's paper \cite{li}.
\begin{prop}
Under the hypothesis of 1.1, there does not exist any   fixed
point free normal subgroup of $\H_c(M)$.
\end{prop}
\begin{proof} Suppose that $G$ is a fixed point free normal
subgroup of $H_c(M)$.  It follows that if $M$ has boundary then
$\dim M\geq 2$. Choose a cover $\U\s\B$ such that for any $U\in\U$
there is $f\in G$ such that $U$ and $f(U)$ are disjoint. Take a
cover $\V$ which is starwise finer than $\U$. This is possible
since $M$ is metrizable, so paracompact. We may assume that
$\H_c(M)$ is factorizable with respect to $\V$ (see Def. 5.1(1)
and Prop. 5.2 below). In view of the commutator equalities
$$ [fg,h]=f[g,h]f^{-1}[f,h],\quad [f,gh]=[f,g]g[f,h]g^{-1}
$$
and Theorem 1.1, it follows that
$$\H_c(M)=[\H_c(M),\H_c(M)]=\prod_{U\in\U}[\H_U(M),\H_U(M)].$$
Let $[h_1,h_2]\in[\H_U(M),\H_U(M)]$ with $U\in\U$ and $f\in G$
such that $U\cap f(U)=\emptyset$.  Then
$[h_1,h_2]=[[h_1,f],h_2]\in G$. Thus $\H_c(M)\s G$ as required.
\end{proof}
\noindent\emph{Proof of Corol. 1.3.} $(\Rightarrow)$ It follows
from a theorem of Ling \cite{li} since $\H_c(M)$ is factorizable
(Prop. 5.2 below) and \emph{transitively inclusive}. The latter
means that for any $U,V\in\B$ there is $h\in\H_c(M)$ such that
$h(U)\s V$.

$(\Leftarrow)$ $\H_c(M^o)$ is a normal subgroup of $\H_c(M)$.
 \quad $\square$
\begin{cor} If $\pm\neq\emptyset$, $\H^{\p}_c(M)$ is not simple
\end{cor}

In fact, $\H_c(M^o)$ is a proper normal subgroup of
$\H^{\p}_c(M)$.

\section{Conjugation-invariant norms }
The notion of the conjugation-invariant norm is a basic tool in
studies on the structure of groups.  Let $G$ be a group. A
\wyr{conjugation-invariant norm} (or \emph{norm} for short) on $G$
is a function $\nu:G\r[0,\infty)$ which satisfies the following
conditions. For any $g,h\in G$ \begin{enumerate} \item $\nu(g)>0$
if and only if $g\neq e$; \item $\nu(g^{-1})=\nu(g)$; \item
$\nu(gh)\leq\nu(g)+\nu(h)$; \item $\nu(hgh^{-1})=\nu(g)$.
\end{enumerate}
Recall that a group is called \emph{ bounded} if it is bounded
with respect to any bi-invariant metric. It is easily seen that
$G$ is bounded if and only if any conjugation-invariant norm on
$G$ is bounded.

Let  $g\in[G,G]$. The \emph{commutator length} of $g$, $\cl_G(g)$,
is the least integer $r$ such that $g$ can be expressed by
\begin{equation} g=[h_1,\bar h_1]\ldots[h_r,\bar
h_r]\end{equation}
 for
some $h_i,\bar h_i\in G$, $i=1\ld r$. Observe that the commutator
length $\cl_G$ is a conjugation-invariant norm on $[G,G]$. In
particular, if $G$ is a perfect group then $\cl_G$  is a
conjugation-invariant norm on $G$.  Then $G$ is called
\emph{uniformly perfect} if  $G=[G,G]$  and the norm $\cl_G$ is
bounded.

\begin{dff}
 Let $G$ be a subgroup of $\H(M)$ and let $\B$ be the family of all balls and
 half-balls of $M$.
\begin{enumerate}
\item $G$ is called \emph{factorizable} (resp. with respect to a
cover $\U\s\B$) if for any $g\in G$ there are $d\in\N$, $B_1\ld
B_d\in\B$ (resp. $B_1\ld B_d\in\U$) and $g_1\ld g_d\in G$ such
that
\begin{equation}
g=g_1\ldots g_d\quad\hbox{with}\; g_i\in G_{B_i} \end{equation}
for all $i$. Here $G_B$ is the subgroup of $G$ of all elements
that can be connected to the identity by an isotopy in $G$
compactly supported in $B$. \item Next, a topological group $G$ is
\emph{continuously factorizable} if there exist $d\in\N$,  $B_1\ld
B_d\in\B$, and continuous mappings $S_i:G\r G_{B_i}$,  $i=1\ld r$,
such that for all $g\in G$
\begin{equation*}g=S_1(g)\ldots S_{d}(g).\end{equation*}
 \end{enumerate}
\end{dff}
\begin{prop}
Under the assumption of Theorem 1.1 on $M$, the groups $\hm$ and
$\hdm$ are factorizable with respect to any cover $\U\s\B$. The
same is true for the isotopy groups $\P\hm$ and $\P\hdm$.
\end{prop}
\begin{proof}
If $M$ is compact, it follows from Theorem 2.3. If $M$ admits a
compact exhaustion, the reasoning is similar to that in the proof
of 1.1.
\end{proof}

For any $g\in\hm, g\neq\id$, denote by $\frag_M(g)$ the smallest
$d$ such that (5.2) holds. By definition $\frag_M(\id)=0$. Clearly
$\frag_M$ is a norm on $\hm$.  Likewise we define $\frag_M^{iso}$
 on the isotopy group $\mathcal P\hm$. Clearly
$\frag_M(f)\leq\frag_M^{iso}(f_t)$ if $f_t$ is an isotopy
connecting $f$ with the identity.

The significance of $\frag_M$ is illustrated by Theorem 1.4.

\begin{dff}
\begin{enumerate}
\item A topological group $G$ is  \emph{continuously perfect} if
there exist $r\in\N$ and continuous mappings $S_i:G\r G$, $\bar
S_i:G\r G$, $i=1\ld r$, satisfying the equality
\begin{equation}g=[S_1(g),\bar S_1(g)]\ldots[S_{r}(g),\bar
S_{r}(g)]\end{equation} for all $g\in G$. \item Let $H$ be a
subgroup of $G$. $H$ is said to be \emph{continuously  perfect in
G} if there exist $r\in\N$ and continuous mappings $S_i:H\r G$,
$\bar S_i:H\r G$, $i=1\ld r$, satisfying the equality (5.3)
 for all $g\in H$. Then $r_{H,G}$ denotes the smallest $r$ as
 above.
 \end{enumerate}
\end{dff}
Of course, every continuously perfect group is uniformly perfect.

\begin{prop} Suppose that the closure of  $B$ is included in $U$, where $B$ is a ball (or a
half-ball and $n\geq 2$) and $U$ is open in $M$. Then $\H_B(M)$ is
continuously perfect in $\H_U(M)$ with $r_{\H_B(M),\H_U(M)}=1$.
\end{prop}

\begin{proof} It suffices to observe that in the proof of Lemma
3.1 the homomorphism $S:\H_B(M)\r\H_U(M)$ is continuous, and the
mapping $\bar S$ is a constant depending on $B$ and $U$.
 \end{proof}

The following fact is a consequence of Prop. 5.4.
\begin{prop}
If $\H_c(M)$ is continuously factorizable then it is also
continuously perfect.
\end{prop}
\begin{proof}
If $B_1\ld B_d\in\B$ is as in Def. 5.1(2), then choose any open
subsets $U_1\ld U_d$ with $\overline B_i\s U_i$. Then we use Prop.
5.4 to each pair $(B_i,U_i)$.
\end{proof}

However we do not know whether some homeomorphism groups $\H_c(M)$
are continuously factorizable. See also \cite{ryb} about locally
continuously perfect groups of homeomorphisms.

  Burago, Ivanov and
Polterovich proved in the   \cite{BIP} that $\D(M)$ is bounded
(and a fortiori uniformly perfect) for many manifolds. We will
need some preparatory notions and  results from \cite{BIP}. A
subgroup $H$ of $G$ is called  \emph{strongly m-displaceable} if
there is $f\in G$ such that the subgroups $H$, $fHf^{-1}$\ld
$f^mHf^{-m}$ pairwise commute. Then we say that $f$
\emph{m-displaces} $H$. Fix a conjugation-invariant norm $\nu$ on
$G$ and assume that $H\s G$ is strongly $m$-displaceable. Then
$e_m(H):=\inf\nu(f)$, where $f$ runs over the set of elements of
$G$ that $m$-displaces $H$, is called the \emph{order m
displacement energy} of $H$.

\begin{thm} \cite{BIP} Given
a group $G$ equipped with a conjugation-invariant norm $\nu$ and
given $H\s G$,  if there exists $g\in G$ that $m$-displaces $H$
for every $m\geq 1$ then for all $h\in [H,H]$
\begin{enumerate} \item $\cl_G(h)\leq 2$; and \item $\nu(h)\leq
14\nu(g)$.\end{enumerate}
\end{thm}
It follows from (1) a weaker version of Lemma 3.1.
\begin{cor}
Suppose that  $B$ is a ball and $\overline B\s U$, where $U$ is
open. Then any homeomorphism supported in $B$ can be written as a
product of two commutators of elements of $\H_U(M)$.
\end{cor}
However, contrary to Lemma 3.1, the method based on Theorem 5.6(1)
is still true in the smooth category.

\section{Boundedness of $\hm$ and $\hdm$}
The proof of the following theorem is essentially  in \cite{BIP}.

\begin{thm}
Let $B$ be a ball or a half-ball in $M$ (in the latter case we
assume $n\geq 2$). Then $\H_B(M)$ is bounded.
\end{thm}

For the proof we need the following

\begin{prop} \cite{BIP}
Suppose that $U,V$ are open disjoint subsets of $M$ such that
there is $f\in\H_c(M)$ satisfying $\overline{f(U\cup V)}\s V$.
Then $f$ $k$-displaces $\H_U(M)$ for all $k\geq 1$.

\end{prop}
\begin{proof}Indeed, this follows from the relation $f^k(U)\s
f^{k-1}(V)\setminus f^k(V)$ for all $k\geq 1$.
\end{proof}

\noindent\emph{Proof of Theorem 6.1} We can choose an open subset
$V$ of $M$ disjoint with $B$ and a homeomorphism $f\in\hm$ such
that $\overline{f(B\cup V)}\s V$. In view of Prop. 6.2 $f$
$k$-displaces $\H_B(M)$ for all $k$. Therefore Theorems 1.1 and
5.7(2) imply the assertion. \quad$\square$

\medskip

\noindent\emph{Proof of Theorem 1.4} The part only if is trivial.
Conversely, the proof is an immediate consequence of Prop. 5.2 and
Theorem 6.1 except for the case $n=1$ and $\pm\neq\emptyset$ (see
the proof of 1.5). \quad$\square$
\medskip

Now we turn to the proof of Theorem 1.5. Let $\R_+=[0,\infty)$. We
begin with the following

\begin{prop}
For any decreasing sequence in (0,1) of the form
\[1>b_1>a_1>b_2>a_2>\ldots>b_k>a_k>\ldots>0,\] converging to 0, there exist $f_1,f_2\in \H(\R_+)$
 such that for $k=1,2,\ldots$ one has
\[f_1([a_{2k-1},b_{2k-1}]\cup[ a_{2k}, b_{2k}])\s (
a_{2k}, b_{2k}),\]
\[f_2([a_{2k},b_{2k}]\cup[ a_{2k+1}, b_{2k+1}])\s (
a_{2k+1}, b_{2k+1}).\] Moreover, if we have another sequence
\[1>\tilde b_1>\tilde a_1>\tilde b_2>\tilde a_2>\ldots>\tilde
b_k>\tilde a_k>\ldots>0,\]   then there is an element of $\psi\in
\H(\R_+)$  with $\psi(a_k)=\tilde a_k$ and $\psi(b_k)=\tilde b_k$
for $k=1,2,\ldots$.
\end{prop}

\begin{proof} In order to prove the first assertion it suffices to
choose $f_1$ (and similarly $f_2$) of the form
$\phi=\bigcup_{k=1}^{\infty}\phi_k$ with
$\phi_k([a_{2k-1},b_{2k-1}]\cup[ a_{2k}, b_{2k}])\s ( a_{2k},
b_{2k})$ for all $k$ and with $\supp(\phi_k)$ mutually disjoint.
The $\psi$ in the second assertion is obtained by gluing together
linear homeomorphisms on the consecutive intervals $[a_1,1]$,
$[b_1,a_1]$, and so on.
\end{proof}
\medskip
\noindent\emph{Proof of Theorems 1.4 (for $n=1$ and
$\pm\neq\emptyset$) and 1.5.}  Let $P=\p\t[0,1]$ be a collar
neighborhood embedded in $M$ such that
 $\p$ identifies with $\p\t\{0\}$.
Since $\p=\pm$ is compact, in view of Theorem 2.5 the restriction
mapping $$\hm\ni f\mapsto f|_{\p}\in\H(\pm)$$ is an epimorphism.
It follows from Lemma 1.10 in \cite{BIP} that $\H_c(\pm)$ is
bounded.  Thus it suffices to show the second assertion of Theorem
1.5. Let $g\in \H_c(\p\t\R_+)$. Arguing as in the proof of Theorem
1.2, there is a sequence, converging to 0, of the form
\[1>b_1>\bar b_1>\bar a_1>a_1>b_2>\ldots>b_k>\bar b_k>\bar
a_k>a_k>\ldots>0\] and homeomorphisms $h_1,h_2\in \H_c(\p\t\R_+)$
such that
\begin{equation*}h_1=g\;\hbox{ on }\;\bigcup_{k=1}^{\infty} \p\t[\bar a_{2k-1},\bar
b_{2k-1}],\quad \supp(h_1)\s U_1:=\bigcup_{k=1}^{\infty} \p\t(
a_{2k-1}, b_{2k-1}),\end{equation*} \begin{equation*}h_2=g\;\hbox{
on} \;\bigcup_{k=1}^{\infty} \p\t[\bar a_{2k},\bar b_{2k}],\quad
\supp(h_2)\s U_2:=\bigcup_{k=1}^{\infty} \p\t( a_{2k},
b_{2k}).\end{equation*} Continuing the reasoning from the proof of
Theorem 1.2 for $h'=h^{-1}g$, it can be checked that $g$ admits a
decomposition of the form $$g=h_1h_2h_3h_4,$$ where
$$h_3=g\;\hbox{ on}\; \bigcup_{k=1}^{\infty} \p\t[ b_{2k},
a_{2k-1}],\quad \supp(h_3)\s U_3:=\bigcup_{k=1}^{\infty} \p\t(
\bar b_{2k}, \bar a_{2k-1}),$$ $$h_4=g\;\hbox{ on}
\;\bigcup_{k=0}^{\infty} \p\t[ b_{2k+1}, a_{2k}],\quad
\supp(h_4)\s U_4:=\bigcup_{n=0}^{\infty} \p\t( \bar b_{2k+1}, \bar
a_{2k}),$$ and where $a_0,\bar a_0$ satisfy $1> \bar a_0> a_0>
b_1$.
 Furthermore, $h_j$ satisfy conditions analogous to  (3.4) for $j=1,2,3,4$.

 In view of
 Prop. 6.3 there exist
$\bar f_{j}\in  \H_c^{\p}(\p\t\R_+)$ of the form $\bar f_j=\id\t
f_j$ such that $\H_{U_j}(M)$ is $m$-displaceable by $\bar f_{j}$
for $j=1,2,3,4$ and for all $m\geq 1$.

Let $\nu$ be a conjugation-invariant norm on $\hdm$. In view of
Theorem 5.6(2) and the invariance of $\nu$ we have
\begin{equation*}
\nu(g)\leq\nu(h_1)+\cdots+\nu(h_4)\leq 14(\nu(\bar
f_1)+\cdots\nu(\bar f_4)).\end{equation*}

Observe that the sets $U_1\ld U_4$ depend  on $g$. Nevertheless,
in view of the second assertion of Prop. 6.3 and the invariance of
$\nu$, the norms $\nu(\bar f_j)$ are independent of $g$. It
follows
 that $\nu(g)$ is bounded, as required.
 \quad$\square$

\begin{dff}
A connected open manifold $M$ is called \emph{portable (in the
wider sense)} if there are disjoint open subsets $U$, $V$ of $M$
such that there is $f\in \H_c(M)$ with  $\overline{f(U\cup V)}$
contained in $V$. Furthermore, for every compact subset $K\s M$
there is $h\in \H_c(M)$ satisfying $h(K)\s U$.
\end{dff}

\begin{rem}
The notion of a portable manifold has been introduced in
\cite{BIP} for smooth open manifolds. The definition there is
specific for smooth category and a bit stronger than Def. 6.4 (a
definition similar to 6.4 is also mentioned in \cite{BIP}).
\end{rem}

The class of portable manifolds comprises the euclidean spaces
$\R^n$, the manifolds of the form $M\t\R^n$, or the manifolds
admitting an exhausting Morse function with finite numbers of
critical points such that all their indices are less that
$\frac{1}{2}\dim M$. In particular, every three-dimensional
handlebody is a portable manifold.

\begin{thm}
If $M$ is portable that $\H_c(M)$ is bounded.
\end{thm}

The proof is a consequence of Prop. 5.2, and is completely
analogous to that for diffeomorphisms (Theorem 1.7 in \cite{BIP}).

 \begin{cor}
 If $M^o$ is portable then $\hdm$ is bounded.
\end{cor}
\medskip
The proof follows from Theorems 1.5 and 6.6. In contrast, for
diffeomorphism groups we have  the following
\begin{prop}
Let $M$ be a smooth manifold with boundary  and let $\D^{\p}(M)$
be the subgroup of all $f\in\D(M)$ such that there exists a
compactly supported isotopy $f_t$ with $f_0=\id$ and $f_1=f$
satisfying $f_t|_{\pm}=\id$ for all $t$. Then $\D^{\p}(M)$ is an
unbounded group.
\end{prop}
\begin{proof}
Choose a chart at $p\in\pm$. Then there is  the epimorphism
$$\D^{\p}(M)\ni f\mapsto \jac_p(f)\in\R_+,$$ where $\jac_p(f)$ is the
Jacobian of $f$ at $p$ in this chart. In view of Prop. 1.3 in
\cite{BIP}, an abelian group is bounded if and only if it is
finite. Therefore $\R_+$ is unbounded. Now Lemma 1.10 in
\cite{BIP} implies that $\D^{\p}(M)$ is unbounded.
\end{proof}

\begin{exa}
Let $\bar B^{n+1}\s\R^{n+1}$ be the closed ball and $S^n=\p\bar
B^{n+1}$. Then $\H_c(S^{n})$ is bounded by an argument similar to
that of Theorem 1.11(ii) in \cite{BIP} stating that $\D(S^{n})$ is
bounded. Next, $\H_c(B^{n+1})$ is bounded in view of Theorem 6.6,
where $B^{n+1}$ is the interior of $\bar B^{n+1}$. Hence, due to
Theorem 1.5 the group $\H^{\p}(\bar B^{n+1})$ are bounded.
\end{exa}

\section{The universal covering groups of $\hm$ and $\hdm$}

Let $G$ be  a topological group.   The symbol $\tilde G$ will
stand for the universal covering group of $G$, that is $\tilde
G=\mathcal{P}G/_{ \sim}$, where $\sim$ denotes the relation of the
homotopy relatively endpoints.

We introduce the following two operations on the space of paths
$\P G$. Let $\mathcal{P}^{\star}G=\{ \gamma \in \mathcal{P}G:
\gamma (t)=e \quad \textrm{for} \quad t \in [0,\frac{1}{2}] \}$.
For all $\gamma \in \P G$ we define $\gamma^{\star}$ as follows:

\begin{equation}\nonumber
 \gamma^{\star}(t)=
\left\{
\begin{array}{lcl}
 e& for & t \in [0,\frac{1}{2}]\\
\gamma(2t-1)& for& t \in [\frac{1}{2},1]
\end{array}
\right.
\end{equation}
Then $\gamma^{\star}\in\P^{\star}G$ and  the subgroup $P^{\star}G$
is the image of $\P G$ by the mapping $\star:\gamma\mapsto
\gamma^{\star}$. The elements of $\P^{\star}G$ are said to be
\wyr{special} paths in $G$. It is important that the group of
special paths is preserved by conjugations, i.e. for each $g\in\P
G$ we have $\conj_g(\P^{\star}G)\s\P^{\star}G$ for every $g\in\P
G$, where  $\conj_g(h)=ghg^{-1}$, $h\in\P G$.

Next, let $\mathcal{P}^{\square}G=\{ \gamma \in \mathcal{P}G:
\gamma (t)=\gamma(1) \quad \textrm{for} \quad t \in [\frac{1}{2},
1] \}$. For all $\gamma \in \P G$ we define $\gamma^{\square}$ by

\begin{equation}\nonumber
 \gamma^{\square}(t)=
\left\{
\begin{array}{lcl}
 \gamma(2t)& for & t \in [0,\frac{1}{2}]\\
\gamma(1)& for& t \in [\frac{1}{2},1]
\end{array}
\right.
\end{equation}
As before $\gamma^{\square}\in\P^{\square}G$ and  the subgroup
$\P^{\square}G$ coincides with the image of $\P G$ by the mapping
$\square:\gamma\mapsto \gamma^{\square}$.

\begin{lem}\label{zero}
For any $\gamma \in \P G$ we have $\gamma \sim
{\gamma}^{\star}$ and $\gamma\sim\gamma^{\square}$.
\end{lem}

\begin{proof}
We have to find a homotopy $\Gamma$ rel. endpoints between
$\gamma$ and $\gamma^{\star}$. For all $s\in I$ define $\Gamma^{\star}$ as
follows:
\begin{equation}\nonumber
\Gamma^{\star}(t,s)=
\left\{
\begin{array}{lcl}
e& for & t\in [0,\frac{s}{2}]\\
\gamma(\frac{2t-s}{2-s})& for& t\in (\frac{s}{2},1]
\end{array}
\right.
\end{equation}
It is easy to check that such  $\Gamma^{\star}$ fulfils all the
requirements.\\
For the second claim define $\Gamma^{\square}$ as follows: for any $s \in I$
\begin{equation}\nonumber
\Gamma^{\square}(t,s)=
\left\{
\begin{array}{lcl}
\gamma(\frac{2t}{2-s})& for & t\in [0,\frac{2-s}{2}]\\
\gamma(1)& for& t\in (\frac{2-s}{2},1]
\end{array}
\right.
\end{equation}
\end{proof}

Given a group $G$  recall the definition of homology groups of
$G$. The usual construction of homology groups proceeds by
defining a \emph{standard chain complex} $C(G)$. Its homology is
the homology of $G$.

 The complex $C(G)$ is defined as follows. For
any  integer $r\geq 0$ denote
$$ C_r(G)\quad =\quad \hbox {free abelian group on the set of all $r$-tuples}\,(g_1,\ldots ,g_r),$$
where $g_i\in G$. Next introduce the \emph{boundary operator}
$\partial :C_r(G)\rightarrow C_{r-1}(G)$ by the formula
\begin{equation*}
\partial (g_1,\ldots ,g_r)=(g_1^{-1}g_2,\ldots
,g_1^{-1}g_r)+\sum_{i=1}^r (-1)^i(g_1,\ldots,\hat {g_i},\ldots
,g_r).\end{equation*}

 Then $\partial ^2=0$.
 Let $Z_{r}(G)=\{ c\in C_{r}(G): \partial(c)=0 \}$ and $B_{r}(G)=\{c \in C_{r}(G): (\exists b\in C_{r+1}(G)),\,
  \partial (b)=c  \}$.
The symbol $H_r(G)=Z_{r}(G)/B_{r}(G)$ will stand for the $r$-th homology group of
the above chain complex. It is well known that
$$  H_1(G)\,=\,G/[G,G],$$
that is, the first homology group is equal to the abelianization
of $G$. For any $g \in G$ the conjugation mapping $\conj_g:G\r G$
induces an identity so $(\conj_{g})_{\star}(h)=h$ for any $h \in
H_{r}(G)$, c.f. \cite{bro}.

 For any $g\in\P G$ denote $\tilde
g:=[g]_{\sim}\in\tilde G$ and for any $c\in C_r(\P G)$ of the form
$c=\sum k_j(g_{1j},\ldots ,g_{rj})$, where $k_j\in \Z$, denote by
$\tilde c:=\sum k_j(\tilde g_{1j},\ldots ,\tilde g_{rj})$ the
corresponding element of $C_r(\tilde G)$. Then it is easily
checked that
\begin{equation}
\tilde\p\tilde c=[\p c]_{\sim}=\widetilde{\p c},
\end{equation}
where $\tilde \p$ is the differential in the chain complex
$C_r(\tilde G)$. That is, (7.1) can serve as a definition of
$\tilde\p$.

In order to compute $H_r(\H_c(\R^n)^{\sim})$ we fix notation. Let
$c=\sum k_j(g_{1j},\ldots ,g_{rj})$, where $k_j\in \Z$, be a chain
from $C_r(\P\H_c(\mathbb{R}^n))$. We define the support of $c$ by
$$  \supp(c):=\, \bigcup _{i,j} \supp(g_{ij}), $$
where $\supp(g):=\,\bigcup _{t\in I} \supp(g_{t})$,  for $g:I\ni
t\mapsto g_t\in \hrn$. Thus $\supp(c)\subset U$ iff
$\supp(g_{ij})\subset U$ for each $i,j$, or
$(g_{ij})_t\in\H_U(\R^n)$ for each $i,j,t$.

\begin{thm}\label{homos}  Let $G$ be either $\hrn$, or $\H_c(\R^n_+)$
 (in the cases $\R^n_+$ we assume $n\geq 2$).
For $r\geq 1$ one has $H_r(\tilde G)=0$. In particular, $\tilde G$
is a perfect group.
\end{thm}

Let $B\s\R^n$ be a ball or $B\s\R^n_+$ be a half-ball. By $\iota
:\H_B(\R^n)^{\sim}\rightarrow \H_c(\R^n)^{\sim}$ we denote the
inclusion, and $\iota _*: H_r(\H_B(\R^n)^{\sim})\rightarrow
H_r(\H_c(\R^n)^{\sim})$ is the corresponding map on the homology
level.

\begin{lem}\label{lem2}
$\iota _*$ is an isomorphism.
\end{lem}

\begin{proof}
 First we show that $\iota _*$ is surjective.
 Let $h \in H_{r}(\H_c (\R^n)^{\sim})$ and let $h=\tilde c$, where $c=\sum k_j(g_{1j},\ldots ,g_{rj})$
  be a cycle representing $h$.
  According to  Lemma \ref{zero} we can assume that $g_{ij}\in\P^{\star}\H_c(\mathbb{R}^{n})$.
  Then $C=\supp(c)$ is compact. We can find
 $\bar{\varphi} \in \P\H_c(\mathbb{R}^{n})$ such that
 $\bar{\varphi}_1(C) \subseteq B$. Define $\varphi:=\bar\varphi^{\square} \in \P^{\square}\H_c(\mathbb{R}^{n})$.
  Since any conjugation induces the
 identity on homology, $(\conj_{\varphi})_{*}(h)=h$. But, in view of
 (7.1), $(\conj_{\varphi})_{*}(h)$ is represented
 by the cycle $\conj_{\varphi}(c)$. It is easily seen that for
 $0\leq t\leq \frac{1}{2} \quad \conj_{\varphi}(c)_{t}=\id$,
  and for $\frac{1}{2}\leq t\leq 1 \quad \conj_{\varphi}(c)_{t}$ is supported in $B$.
  Hence $\conj_{\varphi}(c)$ is a the cycle representing homology
  $h'$ of the group $\H_B(\R^n)^{\sim}$ such that $\iota_*h'=h$.

In order to show injectivity let $h \in \ker (\iota_{*})$. As
above let $c$ be a cycle from
$\mathcal{P}^{\star}\H_c(\mathbb{R}^{n})$ representing $h$.  Since
$\iota_{*}(h)=0$, there is a cycle $c'\in
C_{r+1}(\P\H_c(\mathbb{R}^{n}))$ such that $\partial{c'}=c$. In
view of Lemma \ref{zero}  we may assume that $c'\in
C_{r+1}(\P^{\star}\H_c(\mathbb{R}^{n}))$. We choose $\varphi \in
\P^{\square}\H_c(\mathbb{R}^{n})$ such that
$\varphi_1(\supp(c'))\subseteq B$. We then have
$\partial(\conj_{\varphi}(c'))=\conj_{\varphi}(\partial
c')=\conj_{\varphi}(c)=c$. This means that $c$ is the boundary of
an element from  $C_{r+1}(\P\H_{B}(\mathbb{R}^{n}))$.
Consequently, $h=0$.
\end{proof}

\noindent \emph{ Proof of Theorem 7.2}. (See also \cite{Mat71}.)
By Lemma 7.3 it suffices to consider $\H_B(\R^n)$ (resp.
$\H_B(\R^n_+)$), where $B\s\R^n$ is a ball (resp. $B\s\R^n_+$ is a
half-ball). As in the proof of Lemma 3.1 we define $B_0=B$ and we
choose a locally finite, pairwise disjoint sequence of balls
(resp. half-balls) $(B_k)_{k=0}^{\infty}$ converging to a point
$p\in\R^n$ (resp. $p\in\p_{\R^n_+}$). We also choose an isotopy
$\phi\in\P^{\square}\H_c(\R^n)$ (resp.
$\phi\in\P^{\square}\H_c(\R^n_+)$) with $\phi_1(B_k)=B_{k+1}$ for
$k=0,1,\ldots$.

 Now define
$\psi_{i}:\P^{\star}\H_{B}(\mathbb{R}^{n})\rightarrow
\P^{\star}\H_c(\mathbb{R}^{n})$ for $i=0,1$ as follows: for each
class $g \in \P^{\star}\H_B(\mathbb{R}^{n})$ for each $t \in
[0,1]$ we put:

\begin{equation}\nonumber
 \psi_{i}(g)(x)=
   \left\{
      \begin{array}{lcl}
           \varphi^{j}g\varphi^{-j}(x) & for & x \in    \bar{\varphi}^{j}(\overline{B}), j\geq i\\

           x& for & x \notin \bigcup_{j\geq i} \bar{\varphi}^{j}(\overline{B})
      \end{array}
   \right.
\end{equation}

It is obvious that $\psi_{0}$ and $\psi_{1}$ are conjugate so
$(\psi_{0})_{*}=(\psi_{1})_{*}$. Now define
$\eta:\P^{\star}\H_{B}(\mathbb{R}^{n})\times
\P^{\star}\H_{B}(\mathbb{R}^{n})\rightarrow
\P^{\star}\H(\mathbb{R}^{n})$ as
$\eta(g,h)(t)=g_{t}\psi_{1}(h_{t})$.
 It is easy to prove that $\eta$ induces $\bar\eta:
 \H_{B}(\mathbb{R}^{n})^{\sim}\times \H_{B}(\mathbb{R}^{n})^{\sim}\rightarrow \H(\mathbb{R}^{n})^{\sim}$
  since if $g\sim \bar{g}$ and $h \sim \bar{h}$ then also $\eta(g,h) \sim \eta(\bar{g},\bar{h})$.
Let $\Delta:\H_B(\mathbb{R}^{n})^{\sim}\rightarrow
\H_B(\mathbb{R}^{n})^{\sim} \times \H_B(\mathbb{R}^{n})^{\sim}$ be
the diagonal map. Then \begin{equation}\psi_{0}=\eta \Delta.
\end{equation} Now we proceed by the induction on $r$. For $r=0$ the
assertion is trivial. For the inductive step we may assume that
$H_s(\H_B(\R^n)^{\sim})=0$ for $1\leq s \leq r-1$. Then by the
Kunneth formula we
 get
\begin{equation}  H_r(\H_B(\R^n)^{\sim}\times \H_B(\R^n)^{\sim})\,
=\, H_r(\H_B(\R^n)^{\sim}) \oplus H_r(\H_B(\R^n)^{\sim}).
\end{equation}

Now choose arbitrarily $\{c\}\in H_r(\H_B(\R^n)^{\sim})$. Then
$\Delta _*\{c\}= \{c\}\oplus \{c\}$ by (7.3). It follows by (7.2)
and (7.3) that \[ \psi _{0*}\{c\}=\eta _*\Delta _*\{c\}=\iota
_*\{c\}+\psi _{1*}\{c\} =\iota _*\{c\}+\psi _{0*}\{c\}.\] Thus
$\iota _*\{c\}=0$, and $\{c\}=0$ by Lemma 7.3, as required. In the
case of $\H_c(\R^n_+)$ the proof is the same. \quad $\square$

\medskip

\noindent\emph{Proof of Theorem 1.6}. The second claim coincides
with Theorem 7.2. The first claim is a consequence of Prop. 5.2
for $\P\H_c(M)$, and of the second claim.  \quad$\square$
\begin{thm}
Let $B$ be a ball or a half-ball in $M$ (in the latter case we
assume $n\geq 2$). Then $\H_B(M)^{\sim}$ is bounded.
\end{thm}
\begin{proof}
 Let $f\in\H_c(M)$ be as in Prop. 6.2. We choose an isotopy
 $f_t\in\P^{\square}\H_c(M)$ joining $f$ with the identity. Next
 we observe that, due to Theorem 1.6 and Lemma 7.1 any class from
 $\H_B(M)^{\sim}$ can be represented as a product of commutators
 of elements from $\P^*\H_B(M)$. The proof is now analogous to
 that of Theorem 6.1.
\end{proof}

\noindent\emph{Proof of Theorem 1.7}. It follows from Prop. 5.2
and Theorem 7.4. \quad$\square$

\end{document}